\documentclass[11pt]{amsart}

\newskip\stdskip
 \usepackage{amsfonts}
\usepackage{amsmath}
\usepackage{amssymb}
\usepackage{amstext}
\usepackage{amsthm}

\usepackage[all]{xy}

\usepackage[dvips]{epsfig}
\usepackage{pinlabel}
\usepackage{graphicx}
\usepackage[usenames,dvipsnames]{color}
\usepackage{enumitem}
\usepackage{mathtools}

\usepackage{hyperref}

\newtheorem{thm}{Theorem}[section]
\newtheorem{lemma}[thm]{Lemma}

\theoremstyle{definition}

\newtheorem{rem}[thm]{Remark}

\begin{document}

\author[B. Chantraine]{Baptiste Chantraine}

\address{Universit\'e de Nantes, France.}
\email{baptiste.chantraine@univ-nantes.fr}



\title{Reeb chords of Lagrangian slices}
\maketitle

\begin{abstract}
  In this short note we observe that the boundary of a properly embedded compact exact Lagrangian sub-manifold in a subcritical Weinstein domain $X$ necessarily admits Reeb chords. The existence of a Reeb chord either follows from an obstruction to the deformation of the boundary to a cylinder over a Legendrian sub-manifold or from the fact that the wrapped Floer homology of the Lagrangian vanishes once this boundary has been ``collared''. 
\end{abstract}

\section{Introduction}

Let $Y$ be a contact manifold  with contact form $\alpha$, a \emph{Lagrangian slice} in $Y$ is a sub-manifold $i:\Lambda \hookrightarrow Y$ such that their exists a Lagrangian embedding $$C: (1-\varepsilon,1+\varepsilon)\times \Lambda\hookrightarrow \left((1-\epsilon,1+\epsilon)\times  Y,d(t\alpha)\right)$$ satisfying:

$$C\left((1-\varepsilon,1+\varepsilon)\times \Lambda\right)\pitchfork \{1\}\times Y=C\left(\{1\}\times\Lambda\right).$$

It is the opinion of the author that such objects are interesting and the aim of this note is to attract attention to those. Note that the definition depends on the contact form as the transverse intersection condition depends on the identification of the symplectisation with $\mathbb{R}_+\times Y$.  Any Legendrian sub-manifold of $Y$ is a Lagrangian slice since the trivial cylinder of such a sub-manifold is Lagrangian (that this notion does not depend on the choice of the contact forms comes from the fact that those are the slices giving Lagrangians tangent to the Liouville vector field, they are therefore transverse to any hypersurface transverse to this vector field).  Naturally Lagrangian slices appear in the following manner : let $L$ be a Lagrangian sub-manifold of a symplectic manifold $(M,\omega)$. Let $Y$ be a contact hyper-surface of $M$ such that $L\pitchfork Y$. Then $\Lambda=Y\pitchfork L$ is a Lagrangian slice. A \emph{Reeb chord} of a Lagrangian slice $\Lambda$ is a trajectory of the Reeb vector field of $\alpha$ whose start and end points are on $\Lambda$.

There is a seemingly related notion to Lagrangian slices which are prelagrangian submanifolds (studied for instance in \cite{MR1334868}). A \emph{prelagrangian} submanifold of $(Y^{2n-1},\ker\alpha)$ is an $n$ dimensional submanifold $L$ transverse to $\ker\alpha$ such that there is a function $f$ on $Y$ such that the pullback of $f\alpha$ to $L$ is closed. In this situation any preimage of regular values of $f|_L$ is a Lagrangian slice. As embedded projections of Lagrangians in $\mathbb{R}_+\times Y$ to $Y$ are prelagrangian the correspondence goes the other way around (note that one can always arrange a transverse slice to have embedded image in $Y$ playing a bit with the Reeb flow).

A Lagrangian slice $\Lambda$ is said to be \emph{fillable} if there exists a filling $X$ of $Y$ and a properly embedded Lagrangian sub-manifold $L$ of $X$ such that $\partial L=L\pitchfork Y=\Lambda$. Those are the slices that appear when intersecting closed Lagrangians (or Lagrangians cylindrical at infinity in Liouville manifolds) with a \emph{separating} contact hyper-surface. 

If one can deform the Liouville structure on $(1-\varepsilon,1+\varepsilon)\times Y$ keeping $Y$ a contact hyper-surface and so that the Liouville vector field is tangent to $(1-\varepsilon,1+\varepsilon)\times \Lambda$ then we say that the slice is \emph{collarable}. In \cite{collarable} we gave examples of non-collarable fillable slices. Other examples of non-collarable slices appear in \cite{Lin} and \cite{Eliash_Mur_Caps} as a slice of a Lagrangian cap cannot be collarable near its maxima.

\begin{rem}\label{rem:collar}
  One might find this definition of collarable unsatisfying from the contact point of view: the contact structure on $Y$ changes. In the compact (or local) case, Moser-Gray type theorem implies that this is the same as asking that there is an isotopy of $\Lambda$ through Lagrangian slices to a Legendrian sub-manifold of $Y$. We keep the definition this way as it is easier to manipulate, it is more convenient to study Reeb chords of slices, and it relates more easily to the notion of regular Lagrangian.
\end{rem}

The question of collarability of slices is implicit in the question of decomposability of cobordisms (see \cite{Ekhoka} and \cite{MR3493408}) which appears prominently in the literature about Lagrangian cobordisms in low dimension. In all dimension the collarability question is a sub-question of the question of regularity of Lagrangian sub-manifold from \cite{MR4090744} (a Lagrangian is regular if one can modify the Liouville vector field so that it is tangent to the Lagrangian) as in low dimension regular cobordisms are decomposable (as observed for instance in \cite{MR4045364}). The author believes that the quantitative study of Lagrangian slice, particularely Inequality \eqref{eq:2} from Section \ref{sec:coll-lagr-slic}, could provide some tools to study the question of regularity of cobordism.

There is a whole hierarchy of fillability that we can imagine and it is not the purpose of this note to give a comprehensive overview of it. Our aim is to suggest that even if such objects form a larger class than Legendrian sub-manifolds they still admits Reeb chords in a similar fashion as Legendrians do (we do not claim here that all Legendrians admits Reeb chords). For instance we prove the following:

\begin{thm}\label{thm:main}
  Let $\Lambda$ be a compact Lagrangian slice in a compact contact manifold $Y$ such that there exist a subcritical Weinstein manifold $W$ of $Y$ and a proper exact Lagrangian embedding $L\hookrightarrow W$ such that $\partial L=\Lambda$. Then $\Lambda$ admits a Reeb chord.
\end{thm}

In particular, if its sub-level is subcritical, the transverse intersection of a compact  exact Lagrangian with the level set of a convex Hamiltonian always has trajectories of the Hamiltonian vector field starting and ending on the Lagrangian on that level set.

In Section \ref{sec:proof-ot-theorem} we provide a proof of this theorem and it appears to be ridiculously easy (it almost appears in \cite[Section 5.1]{chantraine_conc}), so the interest of the paper lies somewhere else (if it lies anywhere). As mentioned before the only purpose of this note is to raise some interesting points about such objects and discuss some of their Hamiltonian dynamics properties.

\paragraph{\textbf{Acknowledgement}}

Despite the short length of this note I have been casually talking about this observation for quite a while. I want to acknowledge discussions with Vincent Colin, Kyle Hayden, Emmy Murphy and Egor Shelukhin that lead me to write it. I also want to thank an anonymous referee whose comments helped to clarify and improve the paper in many places. I am partially supported by the ANR projects QUANTACT (ANR-16-CE40-0017) and  ENUMGEOM (ANR-18-CE40-0009) 

\section{What does a Lagrangian slice look like?}
\label{sec:how-does-lagrangian}
Since they form a a larger class than Legendrian sub-manifolds it is a natural question to ask what properties a Lagrangian slice must satisfies.

Let $i:\Lambda\rightarrow Y$ be a sub-manifold. If it is a Lagrangian slice (i.e. a transverse intersection of a Lagrangian in the symplectisation with $Y$) then the general (linear) theory of symplectic reduction tells you that the Reeb vector field is not tangent to $\Lambda$ and that $\Lambda$ projects to a (smooth) Lagrangian immersion in the (singular) reduced symplectic space $Y/R_\alpha$. So if it is a Lagrangian slice then we must have $d(i^*\alpha)=0$ and $di(T\Lambda)\cap \ker d\alpha=\{0\}$.

Conversely, let $i:\Lambda\hookrightarrow Y$ a sub-manifold such that the pull-back of $\alpha$ is closed and $di(T\Lambda)\cap \ker d\alpha=\{0\}$. Let $V$ be a copy of a neighborhood of the $0$-section in $T^*\Lambda$ such that $\Lambda$ corresponds to the $0$-section and $V$ is transverse to $R_\alpha$. The characteristic foliation of the hypersurface $(1-\varepsilon,1+\epsilon)\times V$ in $(1-\varepsilon,1+\epsilon)\times Y$ is transverse to $\{1\}\times V$ and crossing $\Lambda$ with this foliation realises $\Lambda$ as a Lagrangian slice. We have proved:

\begin{lemma}
  
  Let $i:\Lambda^{n-1}\rightarrow Y^{2n-1}$ be an embedding. The $\Lambda$ is a Lagrangian slice iff $di(T\Lambda)\cap \ker d\alpha=\{0\}$ and $i^*\alpha$ is closed.
\end{lemma}

\begin{rem}
  Observe that when $Y$ is of dimension $3$ then these conditions reduce to asking that  $\Lambda$ is not tangent to the Reeb foliation which is a generic condition.
\end{rem}
Since for a Lagrangian slice $i^*\alpha$ is closed, one can define the notion of \emph{exact} Lagrangian slice by requiring $i^*\alpha$ to be exact (for instance the non-collarable examples from \cite{collarable}, \cite{Lin} and \cite{Eliash_Mur_Caps} are exact). This implies that the Lagrangian $(1-\epsilon,1+\epsilon)\times \Lambda$ is exact. This notion is invariant under modifications of the Liouville structure by Hamiltonian vector fields.

Explicit occurence of this construction appears in \cite{MR2134231} and \cite{Lin} in the particular case of the standard symplectic $3$ dimensional case where we see that up to translation in the Reeb direction the shape of a Lagrangian is determined by the symplectic reduction of its slices.

\section{Deformation of Liouville vector fields near the boundary.}
\label{sec:deform-liouv-vect}
In this section we study how to deform Liouville vector fields so that they remain transverse to a given contact hyper-surface, such computations are standard and appear in different forms in the literature see for instance \cite[Lemma 12.1]{SteinWeinstein}.

Let $(W,\lambda;f)$ be a Weinstein cobordism with positive boundary $i:Y\hookrightarrow W$ such that the contact form is given by $i^*\lambda$. Let $(1-\epsilon,1]\times Y$ be a collar neighbourhood of $Y$ such that the Liouville flow structure is given by $t\cdot\alpha$. Let $H$ be a function of the form $H(t,x)=\rho(t)h(x)$ for a function $h$ on $Y$ and a compactly supported increasing function $\rho$ on $(1-\epsilon,1]$ such that $\rho(1)=1$. We denote the forms $\lambda_H:=\lambda+dH$, it is still a Liouville form for the symplectic structure on $W$ the Liouville vector field for $\lambda_H$ is $V_\lambda+X_H$. The contact hyper-surface $Y$ is  still transverse to $X_{\lambda_H}$ iff at $t=1$:

$$dt(V_\lambda+X_H)>0$$

Since $dt(V_\lambda)=t$ this condition becomes
$$dt(X_H)>-1.$$

The symplectic form writes as $dt\wedge \alpha+t\cdot d\alpha$ thus we have that $$dH=h\cdot \rho' dt+\rho\cdot dh=dt(X_H)\alpha+t(X_h\iota d\alpha)-\alpha(X_H)dt.$$

Evaluating on $R_\alpha$ (the Reeb vector field of $\alpha$) at $t=1$ we obtain that $Y$ is transverse to $V_{\lambda_H}$ iff
\begin{equation}
dh(R_\alpha)>-1.\label{eq:1}
\end{equation}

\section{Collarability of Lagrangian slices.}
\label{sec:coll-lagr-slic}

Let $i:\Lambda\rightarrow Y$ be an exact Lagrangian slice. We want now to make a deformation similar to the preceding section such that $V_{\lambda_H}$ is tangent to the Lagrangian $j:(1-\epsilon,1+\epsilon)\times\Lambda\hookrightarrow X$ near $\{1\}\times\Lambda$. Let $f$ be a function such that $df=i^*\lambda$ and let $\gamma$ be a Reeb chord of $\Lambda$. We define the \emph{action} of $\gamma$ to be (observe the order of $0$ and $1$ here)
$$a(\gamma)=f(\gamma(0))-f(\gamma(1)).$$

\begin{rem}
  This is a well defined notion whenever $\gamma(0)$ and $\gamma(1)$ are on the same connected component of $\Lambda$ (we refer to such chords as \emph{pure}, the other type being called \emph{mixed}). In the Legendrian case (i.e. when $\Lambda$ is collarable), the action of a chord is $0$ (when it is pure) therefore the action is \emph{not} the action of the corresponding intersection point when taking the Lagrangian projection (when it make sense), this last quantity is what we refer in the present note as the \emph{length} of the Reeb chord.
\end{rem}

In order to have $V_{\lambda_H}$ tangent to the Lagrangian (i.e. symplectically orthogonal to it) we need to have $(j^*\lambda_H)=0$ near $\{1\}\times\Lambda$ which gives
$$j^*\lambda+d(H\circ j)=0.$$

This prescribes the value of $H$ near $\{1\}\times\Lambda$, in particular up to a constant $h=-f$ on $\{1\}\times\Lambda$. To extend $h$ to a function on $Y$ satisfying Inequality \eqref{eq:1} we need to have that any pure chord $\gamma$ of $\Lambda$ satisfies:

\begin{equation}
l(\gamma)>a(\gamma).\label{eq:2}
\end{equation}

Any Reeb chord that violates this inequality will be called \emph{small}, if it is not small a Reeb chord will be called \emph{long}. We have proved

\begin{lemma}\label{lemnotcollar}
  If a Lagrangian slice is not collarable then it has a small (pure) Reeb chords.
\end{lemma}

\begin{rem}\label{rem:stillchords}
  Observe that for any of the deformation $\lambda_H$ of $\Lambda$ as in Section \ref{sec:deform-liouv-vect} the Reeb flow of the new contact form $i^*(\lambda+dH)$ is given by a reparametrisation of the original Reeb flow since $d\alpha$ does not change (an explicit computation give another manifestation of Inequality \eqref{eq:1}). This implies that the question of existence of Reeb chord is unchanged.
\end{rem}

\begin{rem}
  When the Lagrangian slice is not exact one can still try to deform the Liouville form $t\alpha$ adding a \emph{closed} $1$-form. This closed form must extend the restriction of $\alpha$ to the slices so there is first a cohomological condition on the embedding of $\Lambda$ in $Y$. If this condition is satisfied then one can find a cover of $Y$ such that $\alpha$ restricted to the lift of $\Lambda$ is exact and proceed similarely as in the exact case.
\end{rem}

\section{Reeb chords of fillable slices.}
\label{sec:proof-ot-theorem}
We are now ready to prove our main theorem.
\begin{proof}[Proof of Theorem \ref{thm:main}]
  Let $L$ be an exact filling of the slice $\Lambda$. If there are no small Reeb chords then from Section \ref{sec:coll-lagr-slic} we know that we can modify the Liouville vector field so that $L$ becomes an exact filling of a Legendrian sub-manifold. Assuming $\Lambda$ has no chords at all (from Remarks \ref{rem:stillchords} it means that for either of the contact forms $\alpha$ and $\alpha_H$ is has no Reeb chords) then we can positively wrap a small deformation of $L$ without introducing any intersection point outside the domain $X$. We can use this wrapped copy to compute Wrapped Floer homology of $L$ from \cite{OpenString} (either wrapping all at once, or using a direct limit depending what is our favourite definition of $WF(L)$). This implies that no high energy intersection points are involved in the computation of the wrapped Floer homology, therefore it is isomorphic to the infinitesimally wrapped Floer homology i.e. $$WF(L)=H^*(L).$$ However since $W$ is subcritical it has vanishing symplectic homology, $SH(W)=0$ (see \cite{MR1726235}). This is a contradiction as $WF(L)$ is a module over $SH(W)$ (see \cite[Theorem 6.17]{ritter-tqft}) and thus $WF(L)=0$.

  This covers the collarable case, the other case is covered by Lemma \ref{lemnotcollar} and therefore the proof is complete.
\end{proof}

\section{Concluding remarks}
\label{sec:conclusion}

The proof of Theorem \ref{thm:main} applies to more general context depending on the tools one wants to use and the information we have on the slice or the filling to conclude existence of Reeb chords.

Obviously one can replace the subcritical condition of $W$ by the algebraic condition that $SH(W)=0$. Also if $SH(W)\not=0$ one can instead require that the Lagrangian filling $L$ does not intersect the Lagrangian skeleton of $W$ since such a condition guarantees that $WF(L)=0$ by Viterbo's functoriality (see \cite[Theorem 9.11]{CieliebakOancea} or \cite[Section 5]{OpenString}).

We can also hope to relax the vanishing of $SH(W)$ to some more specific case, as long as one can prove that $WF(L)\not = H^*(L)$ we can deduce the existence of the Reeb chords. For that purpose the structural results from \cite{CieliebakOancea} can reveal to be useful.

In another direction we can relax the fillability condition using \cite{Floer_cob} by only requiring that the slice is at the top of a Lagrangian cobordism with bottom being a Legendrian knot that admits an augmentation

It would be interesting if one could use considerations from sections \ref{sec:deform-liouv-vect} and \ref{sec:coll-lagr-slic} to investigate the regular Lagrangian question (and therefore the question of decomposability of Lagrangian cobordisms).

In \cite{MR915560} Arnol'd conjectured that any Legendrian in the standard $3$-sphere admits a Reeb chord for any Reeb vector field. This version was proved in \cite{MR1847594} (for all standard contact spheres) and in \cite{MR2784673} \cite{MR3190296} (for any Legendrian in any 3 dimensional contact manifold). In the discussion following his conjecture he also discusses that some homological bounds might exist for the number of such Reeb chords, some instance of this can be found in \cite{Ekholm_Contact_Homology} and \cite{Rigidityofendo} where some estimate is given in terms of the topology of the Legendrian or some of its Lagrangian fillings. The existence result from Theorem \ref{thm:main} falls in between: the collarable case indeed gives the usual estimate coming from wrapped Floer homology or linearised contact homology but the obstruction to collarability only provides one chord. Whether or not more can be found in the non collarable case is unknown to the author.
\bibliographystyle{plain}
\bibliography{Bibliographie_en}
\end{document}